\documentclass{amsart}
\usepackage{amssymb,amsmath,amsthm}
\begin{document}

\newtheorem{thm}{Theorem}[section]
\newtheorem{cor}[thm]{Corollary}
\newtheorem{lem}[thm]{Lemma}
\newtheorem{prop}[thm]{Proposition}
\newtheorem{defn}[thm]{Definition}

\theoremstyle{definition}
\newtheorem{rmk}[thm]{Remark}

\newcommand{\End}{\operatorname{End}}
\newcommand{\Ker}{\operatorname{Ker}}
\newcommand{\Disc}{\operatorname{Disc}}
\newcommand{\rec}{\operatorname{rec}}
\newcommand{\Pic}{\operatorname{Pic}}
\newcommand{\Gal}{\operatorname{Gal}}
\newcommand{\Aut}{\operatorname{Aut}}
\newcommand{\CM}{\operatorname{CM}}
\newcommand{\sign}{\operatorname{sign}}

\title[On the independence of Heegner points\dots]{On the independence of Heegner points on CM elliptic curves associated to distinct quadratic imaginary fields}

\author{Hatice \c Sahino\u glu}
\date{\today}
\begin{abstract}In this paper we give a sufficient condition on the class numbers of distinct quadratic imaginary fields so that on a given CM elliptic curve over $\mathbb{Q}$ with fixed modular parametrisation, the Heegner points associated to (the maximal orders of) these quadratic imaginary fields are linearly independent. This extends results of Rosen and Silverman from the non-CM to the CM case.
\end{abstract}
\maketitle
\section{Introduction}
In this paper we give a result for CM elliptic curves similar to the one in the paper \cite{RS} by Michael Rosen and Joseph H. Silverman. In their paper, they have considered a non-CM elliptic curve $E/\mathbb{Q}$ with a given modular parameterization $\Phi_{E}:X_0(N)\rightarrow E$. Let $P_1,P_2, P_3,...P_r \in E(\bar{\mathbb {Q}})$ be Heegner points associated to the ring of integers of distinct quadratic imaginary fields $k_1,k_2, k_3,...k_r$. If the odd parts of the class numbers of $k_1,k_2, k_3,...k_r$ are larger than a constant $C=C(E,\Phi_{E})$ (depending only on $E$ and $\Phi_{E}$), then they prove that the points $P_1,P_2, P_3,...P_r$ are independent in $E(\bar{\mathbb{Q}})/E_{tors}$. In this  paper we show that the result holds if~$E$ has CM.
We follow the same procedure as in the paper \cite{RS} and we use  most of the results from there. This makes \cite {RS} our main reference, and the details of many steps will be skipped by referring to \cite{RS}. The main difficulty for the CM case is that Serre's theorem on the image of Galois does not  apply, which obliges us to find another technique to estimate a uniform bound for the prime-to-$2$ part of $[k_i(P_i):k_i(nP_i)]$, for all $n$ and $i$, where $1\leq i \leq r$.
Here is the exact statement of our main result, which includes both our CM case and the Rosen--Silverman non-CM case.
\begin{thm}\label{main2}
Let $E/\mathbb{Q}$ be an elliptic curve and let
$$\Phi_{E}:X_0(N)\rightarrow E$$ be a modular parametrization.
Assume that we are given the following:
\begin{itemize}
 \item $k_1,k_2, k_3,...k_r$ are distinct quadratic imaginary fields satisfying the Heegner condition for N,
 \item $h_1,h_2, h_3,...h_r$ are the class numbers of $k_1,k_2, k_3,...k_r$,
 \item $y_1,y_2, y_3,...y_r$ are points on the modular curve associated to the ring of integers of $k_1, k_2, k_3,...k_r$ respectively,
 \item $P_1,P_2, P_3,...P_r$ are the associated Heegner points $P_i=\Phi_{E}(y_i)$.
 \end{itemize}
There exists a constant $C=C(E,\Phi_E)$, which
depends only on $E$ and the modular parametrization $\Phi_E$, such that if the odd part of the class numbers of $k_1,k_2, k_3,...k_r$ are larger than $C$,
then the points $P_1,P_2, P_3,...P_r$ are independent in $E(\bar{\mathbb{Q}})/E_{tors}$
\end{thm}
After presenting the proof of the Theorem \ref{main2}, we will generalize the result to Heegner points associated to  non-maximal orders of fixed conductor.
\section{Heegner Points}
We consider Heegner points on the modular curve $X_{0}(N)$ and their images, which also are  called Heegner points, under a parametrization of an elliptic curve. In this section we give some properties of Heegner points on modular curves and on elliptic curves. We do not give the proofs, but we give references for the results that we quote.

\subsection{Heegner points on a Modular Curve of Level N}
The points of the non-cuspidal part of the modular curve $X_{0}(N)$ are in one-to-one correspondence with isomorphism classes of triples $(A, A', \phi)$, where $A, A'$ are elliptic curves and $\phi$ is an isogeny from $A$ to $A'$ whose kernel is cyclic of order $N$.
Heegner points are points on a modular curve that are associated to orders $\mathcal{O}$ in quadratic imaginary fields $k$. They  will be the points corresponding to $(A, A', \phi)$, where
$$\End(A)\cong \End(A') \cong \mathcal{O},$$
and $\phi: A\rightarrow A'$ is an isogeny with $\ker(\phi)\cong \mathbb {Z}/N\mathbb{Z}$. We denote this set by $\CM^{(N)}(\mathcal{O})$. One should note that not all orders give non-trivial points on the modular curve $X_{0}(N)$. The interesting case occurs when the order $\mathcal{O}$ of the ring of integers of $k$ satisfies the Heegner condition for $N$. That is:

\begin{itemize}
\item[(1)] $\gcd(\Disc (\mathcal{O}), N)=1$.
\item[(2)]Every prime dividing $N$ is split in $k$.
\end{itemize}

The following theorem gives some important properties of Heegner points. Before stating it, we introduce ring class fields associated to orders in quadratic imaginary fields. Note that until we start generalizing our results to arbitrary orders, when we mention Heegner points, we will be talking about Heegner points associated to maximal orders of quadratic imaginary fields.

\begin{defn} Let $k$ be a quadratic imaginary field and $\mathcal{O}$ an order in $k$.  The field extension $K_{\mathcal{O}}$ is the field such that
the reciprocity map gives an isomorphism
$$ \rec: \Pic{\mathcal{O}} \cong \Gal(K_{\mathcal{O}}/k).$$
This field is called the \emph{ring class field associated to $\mathcal{O}$}.
\end{defn}

\begin{thm}
Let $\mathcal{O}$ be an order in $k$ that satisfies the Heegner condition, and let $K_{\mathcal{O}}$ be the ring class field of $k$ associated to $\mathcal{O}$. Then we have the following:
\item[(a)]The points in $\CM^{(N)}(\mathcal{O})$ are defined over $K_{\mathcal{O}}$.
i.e.,
$$\CM^{(N)}(\mathcal{O})\subset X_{0}(N)(K_{\mathcal{O}}). $$
\item[(b)]The set of triplets $(A, A', \phi)$ are in one-to-one correspondence with the set of pairs
$$\{(\mathfrak{n},\bar{a}): \bar{a}\in \Pic(\mathcal{O}),\; \mathfrak{n} \text{ is a proper $\mathcal{O}$-ideal, and }  \mathcal{O}/\mathfrak{n}\cong\ \mathbb {Z}/N \mathbb{Z}\}.$$
The correspondence is given by assigning the pair $(\mathfrak{n},\bar{a})$ to the isogeny
$$ \mathbb{C}/a \rightarrow \mathbb{C}/a\mathfrak{n}^{-1}.$$
\item[(c)]There is a natural action of $\Pic(\mathcal{O})$ on isomorphism classes of triplets (and so on $\CM^{(N)}(0)$) which we denote by $\star$. The action is defined as follows on the pairs $(\mathfrak{n}, \bar{a})$,
$$ \bar{b} \star (\mathfrak{n}, \bar{a})=(\mathfrak{n}, \bar{a}\bar{b}).$$
\item[(d)]Recalling the reciprocity map that gives an isomorphism between $\Pic(O)$ and $\Gal(K_{\mathcal{O}}/k)$, the $\star$-action is compatible with the action of Galois via the reciprocity map in the sense that
$$y^{(\bar{b},K_{\mathcal{O}}/k)}=\bar{b}^{-1}\star y.$$
\end{thm}

\begin{proof}See \cite[chapter 3]{Darmon}.
\end{proof}
This theorem implies the following important result.
\begin{cor}\label{corollary1}
Let $\mathcal{O}$ be an order in $k$ that satisfies the Heegner condition, and let $y\in \CM^{(N)}(\mathcal{O})$. Then
$$k(y)=K_{\mathcal{O}}.$$
\end{cor}

\begin{proof}
See \cite[Corollary 5]{RS}.
\end{proof}

\subsection{Heegner points on elliptic curves}
The theorem of Wiles et al.\ on modularity of elliptic curves \cite{Wi95} says that for any elliptic curve $E/\mathbb{Q}$ there exists a modular parametrization
$$\Phi_{E} : X_{0}(N)\rightarrow E.$$
$\Phi_{E} $ is a finite covering map defined over $\mathbb{Q}$.
Heegner points on the elliptic curve are defined to be images under $\Phi_E$ of Heegner points on the modular curve. Therefore the set of Heegner points of $E$ associated to an order $\mathcal{O}$ is the set
$$\{ \Phi_{E}(y):y\in \CM^{(N)}(\mathcal{O})\}. $$
Let $P_{y}=\Phi_{E}(y)$. The $P_y$'s have coordinates in $K_{\mathcal{\mathcal{O}}}$, and one can investigate $[k(P_{y}):k]$ using the following proposition.
\begin{prop}\label{proposition1}
Let $\Psi :S\rightarrow T$ be a finite map of algebraic curves defined over a number field F, and let $s\in S$ and $t=\Psi (s)$. Then
$[F(s):F]$ divides $[F(t):F]\cdot(\deg\Psi)!$. In particular, when $\mathcal{O}$ is an order in $k$ that satisfies the Heegner condition for N, if we let $y\in \CM^{(N)}(\mathcal{O}) $ and $P_{y}=\Phi_{E}(y)$, then
$[K_{\mathcal{O}}:k]$ divides $[k(P_y):k]\cdot(\deg \Phi_E)!$.
\end{prop}

\begin{proof}
See \cite[Proposition 6]{RS}
\end{proof}

\subsection{Multiples of points and abelian extensions}
We begin this section with the following lemma.
\begin{lem} \label{lemma1}
Let $E$ be an elliptic curve defined over a number field $k$, and let $P$ be a point on the elliptic curve such that    $k(P)/k(nP)$ is a Galois extension.  Then
\begin{equation}\label{divisiblex}
[k(P):k(nP)]  \leq \| \{T: T\text { is an $n$-torsion point in $E(k(P))$}\}\|,
\end{equation}
and hence
\begin{equation}\label{divisible}
[k(P):k(nP)]  \text{ divides }\| \{T: T\text { is an $n$-torsion point in $E(k(P))$}\}\|  !\, .
\end{equation}
\end{lem}

\begin{proof}
We will show that $\Gal({k(P)}/k(nP))$ inject into $E(k(P))[n]$, which will imply that the order of the extension $[k(P):k(nP)]$ is less than the number of $n$-torsion points in $E(k(P))$. Define a map
$$\kappa  : \Gal(k(P)/k(nP)) \rightarrow E[n],$$ where for any element  $\sigma \in \Gal({k(P)}/k(nP))$,
$$\kappa(\sigma)=\sigma(P)-P.$$
Note that $\kappa(\sigma)$ is an $n$-torsion point, since $n (\sigma(P)-P)=\sigma(nP)-nP=nP-nP=0$, and note that $\kappa(\sigma)$ is in $E(k(P))$, since both $\sigma(P)$ and $P$ are there. To show the injectivity of the map, let $\sigma$, $\tau\in\Gal({k(P)}/k(nP))$ and assume that $\kappa(\sigma)=\kappa(\tau)$. This means that
\begin{align*}
\sigma(P)-P=\tau(P)-P \quad\Longrightarrow\quad
\sigma(P)=\tau(P)\quad\Longrightarrow\quad
\sigma=\tau.
\end{align*}
This implies that $[k(P):k(nP)]\leq E(k(P))[n]$, which gives~\eqref{divisiblex}, and then the divisibility property~\eqref{divisible} is immediate. Note that $\kappa$ is not a homomorphism; it is a $1$-cocycle.

\end{proof}

The following is the (easy) CM version of Serre's deep theorem on the
image of Galois. It is a well-known consequence of the theory of~CM.

\begin{thm}\label{serre}
Let $k$ be a number field, and $E/k$ be an elliptic curve with complex multiplication by an order $\mathcal{O}$ of a quadratic imaginary field $K_E$.  For a prime number $p$, we have a representation
$$\rho: \Gal(\bar{k}/k)\longrightarrow \Aut_{\mathcal{O}}E[p]= \left( \mathcal{O}/p\mathcal{O} \right)^{*}$$
defined by the formula
$$\rho(\sigma)\star T=T^{\sigma}\quad\text{for  $\sigma \in \Gal(\bar{k}/k)$.}$$
For all sufficiently large primes~$p$, the representation~$\rho$ is
surjective.
\end{thm}

The next proposition is a key step in the proof of our main theorem.

\begin{prop} \label{ramification}
Suppose given an elliptic curve $E/l$ with CM by $K_E$, where $l$ is defined to be the composition of a quadratic imaginary field $k$ with $K_E$, $p$ a rational prime  such that $p$ is relatively prime to the primes of bad reduction of $E$, and $P$ a point on an elliptic curve $E$ with the property that $l(P)/l(pP)$ is a Galois extension and abelian. If $p$ is sufficiently large and $l(P)/l(pP)$ is non-trivial, then the field extension $l(P)/l$ is ramified at a prime $\wp$ above $p$ in $l(P)$.
\end{prop}

\begin{rmk}
\label{notinj}
Denoting the set of $p^r$-torsion points on $E$ by $E[p^r]$, note that $E[p^r]$ can not inject into $E(\bar{\mathbb{F}}_p)[p^r]$, since the former is isomorphic to $\mathbb{Z}/p^r\mathbb{Z}\times\mathbb{Z}/p^r\mathbb{Z}$, while the latter is either $\mathbb{Z}/p^r\mathbb{Z}$ or $0$.
\end{rmk}

\begin{proof}
Assume that $l(P)/l(p^P)$ is non-trivial, and let $\sigma$  in $\Gal(l(P)/l(pP))$ be a non-trivial element. Then by the Lemma \ref{lemma1},  $\sigma(P)-P$ is a non-trivial torsion point, let's call it $T_1$. Also by \cite[Section2, Chapter2]{Silverman2}, $\Aut _{( \mathcal{O}/p\mathcal{O} )}(E[p])=\left ( \mathcal{O}/p\mathcal{O} \right ) ^{*}$, and $\left ( \mathcal{O}/p\mathcal{O} \right)^{*}$ has at least $(p-1)^2$ elements. Given $T_1$ in $E(l(P))[p]$, we aim to get a point that is linearly independent from $T_1$. Note that multiplication by real integers keeps us in the family of points generated by  $T_1$, and any non-real integral element in $\mathcal{O}$ which is not zero modulo $p$ takes us out of the span of $T_1$ and thus finds a non-trivial torsion point out of that span. But there are only $p-1$ real elements in $\mathcal{O}$ modulo $p$, hence for primes greater than $2$ there will be at least one element $\lambda$ of the automorphism group of $p$-torsion points which will give an independent point by acting on $T_1$.  Observe that, for such $\lambda$, the point $[\lambda](x,y)$ is in $E(l(P))$, since multiplication by $\lambda$ is defined over $K_E$, which is contained in $l$. Hence for any point with coordinates in $l$, multiplying it with an element of the endomorhism ring, we get a point that has coordinates in the same field. By the surjectivity property of the image of Galois representation $\rho$ in Theorem \ref{serre} : for large enough primes, the image of $\rho$ in $\left( \mathcal{O}/p\mathcal{O} \right)^{*}$ is surjective, so there is a Galois element $\gamma$ in $\Gal(\bar{l}/l)$ that  corresponds to the automorphism element $\lambda$ such that $\rho(\gamma) T_1=\lambda (T_1)=T_2$. Hence $\gamma$ can be viewed as an element of $\Gal(l(P)/l)$.  But Remark~\ref{notinj} above says that one of these points will vanish modulo some prime $\wp$ above $p$ in $l(P)$. If $T_1$ is a point that vanishes mod $\wp$, then $\sigma$ turns out to be an element of the inertia group with respect to the prime number above $p$. If $T_2$ is the point that reduces to zero modulo $\wp$, then $\gamma^{-1}T_2\equiv 0 \bmod \gamma ^{-1} \wp$, but $\gamma^{-1}T_2= T_1$, so we have $\sigma(P)-P\equiv 0 \bmod \gamma ^{-1} \wp$. So $l(P)/l$ is a ramified field extension at a prime above $p$.
\end{proof}

\begin{thm}\label{thmM}
Let $E/\mathbb{Q}$ be an elliptic curve with CM, and $B'$ be the product of primes that either ramify in the complex multiplication field or are primes of bad reduction, together with some extra primes that  will be described in the proof. Given any imaginary field $k/\mathbb{Q}$, let $P$ be a point such that $k(P)/k$ is an abelian and unramified field extension. Then there is an integer $M=M(E)$ so that the odd part of $[k(P):k]$ divides $M[k(nP):k]$ for all $n\geq 1$ satisfying $\gcd(n,B')=1$.
\end{thm}

\begin{rmk}
The proof for non-CM elliptic curves can be found in \cite{RS}. The result there is stronger, since the statement is proved not only for $P$ such that $k(P)/k(nP)$ is unramified and abelian, but for any point on the elliptic curve that generates an abelian extension over $k$; see \cite[Theorem 11]{RS} for details. We are proving a weaker theorem for CM elliptic curves that will be enough for our purpose of extending the linear independence result for non-CM elliptic curves to CM elliptic curves.
\end{rmk}

\begin{proof}
Note that
$$[k(P):k]=[k(P):k(nP)]\cdot [k(nP):k].$$
To prove the theorem it is enough to find a bound $m$ for $[k(P):k(nP)]$, since then $M=m!$ will do the job.
Note that if $k(P)/k(nP)$ is a Galois extension, then $[k(P):k(nP)]$ is at most the number of $n$-torsion points in $E(k(P))$, as we showed in Lemma \ref{lemma1}. Given $P$ a point as defined above, the field $k(P)$ is an abelian extension of $k$, so $k(P)/k(nP)$ is a Galois extension when $P$ and $k$ are defined as in the statement of the theorem. Also, the fact $k(P)/k$ is an unramified extension can be used as a restriction when counting the number of $n$-torsion points in $E(k(P))$. Now let's replace the number field $k$ with $l$, which is defined to be $k K_E$, where $K_E$ is the complex multiplication field of $E$. Note that $l$ is a $2$-extension of $k$.  Also note that some primes that ramify in $K_E$ may ramify in the field extension $l(P)/l$. But there are only finitely many such primes, and we will avoid them by including them as factors of $B'$.
In Theorem \ref{ramification} we showed that if $l(P)/ l(pP)$ is non-trivial, then $l(P)/l$ is ramified at a prime above $p$ (assuming that $\gcd(p,B')=1$). $P$ is a point assumed to generate $k(P)$ such that $k(P)/k$ is unramifed. Hence since $l$ is the compositum of $k$ and $K_E$,  the extension $l(P)/l$ cannot ramify outside the primes ramifying in $K_E$. Therefore $l(P)/l(pP)$ can not be non-trivial if $p$ is a prime of good reduction,  not equal to any of the prime numbers that ramify in $K_E$, and does not violate the surjectivity condition in Theorem \ref {serre}.
We note that the odd parts of $[k(P):k(nP)]$  and $[l(P):l(nP)]$ are the same, since $l=k K_E$ and $K_E$ is a quadratic field. Under the assumption that $n$ is prime to $B'$, we can find a bound, say $m$ (which is $1$), for the latter, and then we get the same bound for the odd part of the former.
\end{proof}

Now we give the statement and the proof of a weaker version of our main result. This illustrates the main ideas in the argument. Throughout the proof we mention a few other needed results.

\begin{thm}\label{main1}
Let $E/\mathbb{Q}$ be an elliptic curve, and let
$$\Phi_{E}:X_0(N)\rightarrow E$$ be a modular parametrization.
Given the following:
\begin{itemize}
 \item $k_1,k_2, k_3,...k_r$ are distinct quadratic imaginary fields satisfying the Heegner condition for N,
 \item $h_1,h_2, h_3,...h_r$ are the class numbers of $k_1,k_2, k_3,...k_r$,
 \item $y_1,y_2, y_3,...y_r$ are points on the modular curve associated to the ring of integers of $k_1, k_2, k_3,...k_r$ respectively,
 \item $P_1,P_2, P_3,...P_r$ are the associated Heegner points $P_i=\Phi_{E}(y_i)$.
 \end{itemize}
 Let $B$ be the product of  all primes that are less than or equal to the squares of the prime divisors of the conductor, the primes that ramify in the extension $K_E/\mathbb{Q}$, and the primes excluded in Theorem~\ref{serre}.
There exists a constant $C=C(E,\Phi_{E})$ such that if the prime-to-$B$ part of the class numbers of $k_1,k_2, k_3,...k_r$ are larger than $C$,
then the points $P_1,P_2, P_3,...P_r$ are independent in $E(\bar{\mathbb{Q}})/E_{tors}.$
\end{thm}

\begin{proof}
Let $k=k_1k_2k_3....k_r$, let $K$ be the Hilbert class field of $k$, and let $K_i$ be the Hilbert class field of $k_i$.
As in \cite{RS}, we will assume that $P_1,P_2, P_3,...P_r$ are dependent, and we will prove that this implies an upper bound for the prime-to-$B$ part of that $h_i$'s that depends only on $E$ and $\Phi_{E}$.
Thus we start by assuming that
$$n_1P_1+n_2P_2+n_3P_3+.....+n_rP_r=0$$ with not all $n_i$'s  zero.
Without loss of generality, we assume that $n_r\neq 0$. The dependence relation gives that
$$n_r P_r= -\sum_{i<r}{n_iP_i}\in E \left( \prod_{i<r}{K_i}\right),\text{ since each point }P_i\in E(K_i).$$
Thus,
$$k_r(n_r P_r)\subset K_r \cap k_r \prod_{i<r}{K_i}.$$
From this, we get that
$$[k_r(n_r P_r):k_r]\text{\quad divides\quad}[K_r \cap k_r \prod_{i<r}{K_i}:k_r].$$
We now use \cite[Proposition 18]{RS}, which says that
$$[K_r \cap k_r \Pi_{i<r}{K_i}:k_r]\text{ is a power of $2$}.$$
So the odd part of $[k_r(n_r P_r):k_r]$ must be $1$.
By Proposition \ref{proposition1} and Corollary \ref{corollary1} we have the following equation:
\begin{equation}\label{above}
[k_r(y_r):k_r]=[K_r:k_r]| (\deg \Phi_{E})![k_r(P_r):k_r]=(\deg \Phi_{E})![k_r(P_r):k_r(n_rP_r)][k_r(n_rP_r):k_r].
\end{equation}

Decomposing $n_r$ into factors, say $n_r=p_1^{l_1}p_2^{l_2}...p_t^{l_t}m_r$, where $m_r$ is prime to $B'$ (the product of primes of bad reduction and primes that ramify in $K_E$) and where $p_1, p_2,...p_t$ are prime factors of $B'$, we get
\begin{align*}
[k_r(P_r):k_r(n_rP_r)]=&[k_r(P_r):k_r(m_rP_r)][k_r(m_rP_r):k_r(m_rp_1^{l_1}P_r)][k_r(m_rp_1^{l_1}P_r):k_r(m_rp_1^{l_1}p_2^{l_2}P_r)]...\\
&...[k_r(m_rp_1^{l_1}...p_{t-1}^{l_{t-1}}P_r):k_r(m_rp_1^{l_1}...p_t^{l_t}P_r)].
\end{align*}
Note that for any $m$, each index $[k_r(m P_r):k_r(m p_k^{l_k}P_r)]$ is a product of  $[k_r(m {p_k}^{l}P_r):k_r(m {p_k}^{l+1}P_r)]$, where $l_k-1 \geq l\geq 0$. But each of these indexes is a divisor of $p_k^{2}!$ by Lemma \ref{lemma1}. Taking the prime-to-$B$ part of equation \eqref{above}, at the outmost right hand side we end up with some constant terms and the prime-to-$B$ part of the term $[k_r(P_r):k_r(m_rP_r)]$, which is $1$ by Theorem \ref{thmM}. Thus we get an upper bound for the prime-to-$B$ part of $h_r$ assuming dependency of the Heegner points associated to them. This implies that for a given set of $k_i$'s such that $h_i$'s have prime-to-$B$ parts greater than this number, the  associated Heegner points will be linearly independent. So we are done.
\end{proof}

Note that up to now we have avoided the primes of bad reduction, the primes of $K_E$ that ramify in the field extension $l(P)/l(pP)$, and some other primes that fail the surjectivity property in Theorem \ref{serre}. In the next section we explain how to handle these prime numbers and to improve Theorem \ref{main1} to get Theorem \ref{main2}.

\section{Improving Theorem \ref{main1}}
Let's start with the primes of bad reduction. Since $E/\mathbb{Q}$ is a CM elliptic curve, it has potential good reduction, so let $f$ be a number field over which $E$ has everywhere good reduction. We are interested in bounding $[k(P):k(pP)]$, where $P$ is a point on the elliptic curve that has coordinates in the Hilbert class field of $k$. For this purpose we examine $[lf(P):lf(pP)]$, where $l$ is the compositum of $k$ and the CM field of $E$, which we denote as $K_E$ . From now on we have no problem with primes of bad reduction, but only with the primes of $K_E f$ that ramify in $lf(P)$, and those that fail surjectivity in \ref{serre}.
\begin{thm}\label{ramification2}
Let $p$ be a prime of good reduction for a CM elliptic curve $E/F$, where $F$ is the compositum of three number fields: The $CM$ field $K_E$ of the elliptic curve, a quadratic imaginary field $k$, and a number field $f$ which is a field of good reduction for $E$. Let $P$ be the Heegner point on $E$ associated to the field $k$, or more generally, any point that gives an abelian extension $k(P)/k$. Then, $\left | \Gal(F(P)/F(p^rP) \right |$ is bounded by a constant $N$  that is independent of $r$ and $p$.
\end{thm}

\begin{rmk}
For the case $r=1$, let $\sigma \in \Gal(F(P)/F(pP)$ be non-trivial. Then $\sigma$ is in the inertia group of a prime above $p$ from the proof of Theorem \ref{ramification}. But $F(P)/F(pP)$ is an abelian extension, since its Galois group injects into $\Gal(k(P)/k)$, which is an abelian group, so the inertia groups for primes above the same prime are identical. The size of the inertia group  gives a bound on the size of $\Gal(F(P)/F(pP)$, since each non-trivial element of the Galois group is an element of the inertia group of a prime above $p$, and these groups are identical. Therefore every non-trivial element of $\Gal(F(P)/F(pP)$ is in the inertia group of any fixed prime above $p$ in $F(P)$ . Hence the index $[F(P):F(pP)]$ can not exceed the degree of the inertia group, which is at most of order $[F:\mathbb{Q}]$. For general $r$, the proof is similar to the proof of Theorem \ref{ramification},  but first we need the following version of Theorem \ref{serre}, which is also a
  standard consequence of the theory of CM.
\end{rmk}

\begin{thm}\label{serre2}
Let $k$ be a number field, and let $E/k$ be an elliptic curve with complex multiplication by an order $\mathcal{O}$ of a quadratic imaginary field $K_E$.  For every $n\ge1$ we have a representation (as in Theorem \ref{serre})
$$\rho: \Gal(\bar{k}/k)\longrightarrow \Aut_{\mathcal{O}}(E[n]) \cong \left( \mathcal{O}/n\mathcal{O} \right)^{*}.$$
Then the image of~$\rho$ in $\left( \mathcal{O}/n\mathcal{O} \right)^{*}$ has index that is bounded independent of~$n$. In particular, for large enough prime numbers $p$,  the representation
$$\rho: \Gal(\bar{k}/k)\rightarrow \left( \mathcal{O}/p^r\mathcal{O} \right)^{*}.$$
is surjective for all $r\ge1$.
\end{thm}

\begin{proof}[Proof of Theroem \ref{ramification2}]
Assume that $F(P)/F(p^rP)$ is non-trivial, and let $\sigma$  in $\Gal(F(P)/F(p^rP))$ be a non-trivial element. Then $\sigma(P)-P$ is a non-trivial $p^r$ torsion point, let's call it $T_1$. Also by \cite[Section~2, Chapter~2]{Silverman2},
$$\Aut _{( \mathcal{O}/p^r\mathcal{O} )}(E[p^r])=\left ( \mathcal{O}/p^r\mathcal{O} \right )^{*},$$
and $\left ( \mathcal{O}/p^r\mathcal{O} \right)^{*}$ has more than $p^r-1$  elements. Given $T_1$ in $E(F(P))$, we aim to get a point that is linearly independent from $T_1$. Note that multiplication by real integers will keep us in the span of points generated by  $T_1$, and any non-real integral element in $\mathcal{O}$ which is not zero modulo $p^r$ will take us out of the span of $T_1$, and thus will give a non-trivial torsion point there. But there are only $p^r-1$ real elements modulo $p^r$, hence there will be an element $\lambda$ of the automorphism group of $p^r$-torsion points such that $\lambda(T_1)$ will be a point that is independent point of $T_1$.  Observe that, for such $\lambda$, the point $[\lambda](x,y)$ is in $E(F(P))$, since multiplication by $\lambda$ is defined over $K_E$, which is contained in $F$. Hence for any point with coordinates in $F$, multiplying it by an element of the endemorphism ring gives a point that has coordinates in the same field. We now use the surjectivity of the image of Galois representation $\rho$ in the theorem \ref{serre2}. For large enough prime powers, the image of $\rho$ in $\left( \mathcal{O}/p^r\mathcal{O} \right)^{*}$ is everything, so there is a Galois element $\gamma$ in $\Gal(\bar{F}/F)$ that  corresponds to the automorphism element $\lambda$, i.e., such that $\rho(\gamma) T_1=\lambda (T_1)=T_2$. Hence $\gamma$ may be thought as an element of $\Gal(F(P)/F)$.   But by an earlier remark, one of these points must vanish modulo some prime $\wp$ above $p$ in $F(P)$. If $T_1$ is a point that vanishes mod $\wp$, then
$\sigma$ is an element of the inertia group with respect to the prime $p$. If $T_2$ is the point that reduces to zero modulo $\wp$, then $\gamma^{-1}T_2\cong 0 \bmod \gamma ^{-1} \wp$. But $\gamma^{-1}T_2= T_1$, so we have $\sigma(P)-P=0 \bmod \gamma ^{-1} \wp$. Therefore every $\sigma$ is an element of the inertia group for a prime above $p$, and these groups are same since $F(P)/F(p^rP)$ is abelian. But the ramification degree of $F(P)/F$ is bounded by the degree of $K_Ef/\mathbb{Q}$, say $N$, since $k(P)/k$ is already an unramified field extension, and only the ramification degree can be attained from the ramification degree of $p$ in $K_E f$, which is at most the degree of this field. Hence  the degree of $F(P)/F(p^rP)$ is bounded by some constant $N'$, for all $r$ and $p$, assuming that $E/F$ has everywhere good reduction.
\end{proof}

So we can take care of primes of bad reduction, primes of ramification, and the finitely many primes that were excluded by Theorem \ref{serre} by extending the field $l$ to $lf$, where $f$ is a field of everywhere good reduction, and then using  Theorem \ref{serre2}. By doing this, we showed that for any prime  $p$ in $lf$,
$$[lf(P):lf(p^rP)]\leqq N'$$
and therefore $[k(P):k(p^rP)]$ is bounded, for all $r$, by a constant. This is true for each critical prime that appears as a factor of $B'$ in Theorem \ref{thmM}. But there are only finitely many such primes, and for all other primes $p$ except $2$,  we know that $[k(P):k(p^rP)]$ is $1$ by Theorem \ref{thmM}. So we get a finite bound $N$ for the odd part of $[k(P):k(nP)]$. For the convenience of the reader, we restate our main result (Theorem \ref{main2}) and then give the proof.

\begin{thm}
Let $E/\mathbb{Q}$ be an elliptic curve, and let
$$\Phi_{E}:X_0(N)\rightarrow E$$ be a modular parametrization.
Assume that we are given the following:
\begin{itemize}
 \item $k_1,k_2, k_3,...k_r$ are distinct quadratic imaginary fields satisfying the Heegner condition for N,
 \item $h_1,h_2, h_3,...h_r$ are the class numbers of $k_1,k_2, k_3,...k_r$,
 \item $y_1,y_2, y_3,...y_r$ are points on the modular curve associated to the ring of integers of $k_1, k_2, k_3,...k_r$ respectively,
 \item $P_1,P_2, P_3,...P_r$ are the associated Heegner points $P_i=\Phi_{E}(y_i)$.
 \end{itemize}
There exists a constant  $C=C(E,\Phi_{E})$ such that, if the odd part of the class numbers of $k_1,k_2, k_3,...k_r$ are larger than $C$,
then the points $P_1,P_2, P_3,...P_r$ are independent in $E(\bar{\mathbb{Q}})/E_{tors}.$
\end{thm}

\begin{proof}
Let $k=k_1k_2k_3....k_r$, let $K$ be the Hilbert class field for $k$, and let $K_i$ the Hilbert class field of $k_i$. As in \cite{RS}, we assume that $P_1,P_2, P_3,...P_r$ are dependent, and we prove that this gives an upper bound for the odd part of $h_i$'s that only depends on $E$ and $\Phi_{E}$. Thus suppose that
$$n_1P_1+n_2P_2+n_3P_3+.....+n_rP_r=0$$ with not all $n_i$'s  equal to zero.
Without loss of generality, we may assume that $n_r \neq 0$. The dependence relation gives that
$$n_r P_r= -\sum_{i<r}{n_iP_i}\in E \left( \prod_{i<r}{K_i}\right),$$ since each point $P_i\in E(K_i)$.
Thus,
$$k_r(n_r P_r)\subset K_r \cap k_r \prod_{i<r}{K_i}.$$
From this, we get that
$[k_r(n_r P_r):k_r]$ divides $[K_r \cap k_r \prod_{i<r}{K_i}:k_r]$. Now \cite[Proposition 18]{RS} says that
$[K_r \cap k_r \Pi_{i<r}{K_i}:k_r]$ is a power of $2$, so the odd part of $[k_r(n_r P_r):k_r]$ must be $1$.
By the Proposition \ref{proposition1} and  Corollary \ref{corollary1}, we have
\begin{multline}\label{above2}
[K_r:k_r]| (\deg \Phi_{E})![k_r(P_r):k_r] \\
=(\deg \Phi_{E})![k_r(P_r):k_r(n_rP_r)][k_r(n_rP_r):k_r]\mid (\deg \Phi_{E})! M [k_r(n_rP_r):k_r].
\end{multline}

Also by theorem \ref{ramification2}, the odd part of $[k_r(P_r):k_r(n_rP_r)]$ is known to be bounded by a constant for all $k_r, n_r$. Taking the odd part of \eqref{above2}, we get an upper bound for the odd part of the class numbers of $k_r$'s. This implies that for a given set of $k_r$'s whose $h_r$'s have odd parts greater than a constant, say $C=C(E,\Phi_{E})$, the Heegner points associated to these points are linearly independent.
\end{proof}

\section{On the independence of Heegner points associated to orders of fixed conductor $c$ of quadratic imaginary fields}
We examine  CM elliptic curves and non-CM ones separately, since they have different proofs. We generalize each case separately to Heegner points associated to orders of fixed conductor quadratic imaginary fields.
\subsection{Multiples of points and abelian extensions}
Our goal is to understand the index $[k(P_y):k(nP_y)]$, when $P_y$ is a Heegner point on the elliptic curve $E$. Here~$P_y$ is the image of a point $y$ on the modular curve which has CM by an order of conductor $c$ of a quadratic imaginary field $k$, under a fixed modular parametrization $\Phi_E$. We want to find a bound for this index that is independent of $n$. We know by \cite[Theorem 6]{L1} that $k(y)$ is the ring class field, which sits in the ray class field of conductor $c$, and that this field is an abelian extension of $k$ which has ramification only at primes dividing $c$. We will state two different theorems, since the  proof changes depending on whether the elliptic curve has CM.
For non-CM elliptic curves, the Heegner points (and their multiples) associated to an order of a quadratic imaginary field will satisfy the assumptions of \cite[Theorem 11]{RS},  since $k(P_y)$ is in the ring class field which is an abelian extension of $k$. Therefore the same theorem gives a bound for  $[k(P_y):k(nP_y)]$.
\begin{thm}\label{multiplenonCM}
Let $E/ \mathbb{Q}$ be an elliptic curve without CM, and let $d \geq 1$. There is an integer $M=M(E/ \mathbb{Q}, d)$ so that for any number field $k/ \mathbb{Q}$ and $P\in E(\bar {\mathbb{Q}})$ satisfying
$$[k:\mathbb{Q}]\leq d \text{ and }k(P)/k \text { is abelian},$$
the following estimate is true:
$$[k(P):k] \text{ divides }M[k(nP):k]$$
\end{thm}

\begin{proof}
See \cite[Proposition 6]{RS}
\end{proof}

Now we consider the CM version of Theorem~\ref{multiplenonCM}.

\begin{thm}\label{multipleCM}
Let $E/\mathbb{Q}$ be an elliptic curve with CM. Given any imaginary field $k/\mathbb{Q}$ and a Heegner point $P$ associated to an order of conductor $c$ of $k$, there is an integer $M=M(E)$ so that the prime-to-$2$ part of $[k(P):k]$ divides $M[k(nP):k]$ for all $n\geq 1$.
\end{thm}

\begin{proof}
Note that to prove this theorem, it is enough to find a bound $m$ for the prime-to-$2$ part of $[k(P):k(nP)]$, and then set $M=m!$. (Even the prime-to-$2$ part of $m!$ will do the job.) For $P$ given as above, $k(P)/k$ is an unramified abelian extension everywhere except at the prime divisors of $c$. But we know how to handle these finitely many primes from the detailed discussion in the proof of Theorem \ref{ramification2}. Basically, we first extend our field to a field of everywhere good reduction, and then show there that the bound on the ramification degree gives a bound on the size of $\Gal(k(P)/k(p^rP))$.  We can easily find a bound for these groups by using the fact that each $k(P)$ is just some finite extension of the Hilbert class field, and this index can be bounded by a number $Z$ for all $k$ as long as the points $P$ are sitting in a fixed conductor field extension of their associated field.  And for  good primes (unramified and good reduction ones), we have Theorem \ref{ramification}, which tells us that the odd part of this index is $1$. Thus  we get a bound $N$ for the odd part of the index $[k(P):k(nP)]$ that works for all $n$ and $k$.
\end{proof}

Now we can give the proof of the generalization of the result in \cite{RS} and Theorem \ref{main2} to orders of fixed conductor.

\begin{thm}\label{genorder}
Let $E/\mathbb{Q}$ be an elliptic curve, and let
$$\Phi_{E}:X_0(N)\rightarrow E$$ be a modular parametrization.
Assume that we are given the following:
\begin{itemize}
 \item $k_1,k_2, k_3,...k_r$ are distinct quadratic imaginary fields such that the order of conductor $c$ satisfies the Heegner condition for $N$,
 \item $h^c_1,h^c_2, h^c_3,...h^c_r$ are the ring class numbers of of $k_1,k_2, k_3,...k_r$ for the order of conductor $c$,
 \item $y_1,y_2, y_3,...y_r$ are points on the modular curve associated to the orders of conductor $c$ in the rings of integers of $k_1, k_2, k_3,...k_r$,
 \item $P_1,P_2, P_3,...P_r$ are the associated Heegner points $P_i=\Phi_{E}(y_i)$.
\end{itemize}
There is a number $C=C(E,c,\Phi_{E})$ such that if the odd parts of the ring class numbers $h^c_1,h^c_2, h^c_3,...h^c_r$ are larger than $C$,
then the points $P_1,P_2, P_3,...P_r$ are independent in $E(\bar{\mathbb{Q}})/E_{tors}$.

\end{thm}
\begin{proof}
Let $k=k_1k_2k_3....k_r$,  let $K$ be the Hilbert class field for $k$, let $K_i$ be the Hilbert class field of $k_i$,  let $\tilde {K}_i^{c}$ be the ring class field associated to the order of conductor $c$ in $k_i$, and let $K_i^c$ be the ray class field of conductor $c$.
As in \cite{RS}, we will assume that $P_1,P_2, P_3,...P_r$ are dependent and  will use this to give an upper bound for the odd parts of the $h^c_i$'s that only depends on $E$ and $\Phi_{E}$ and $c$.
Assume that
$$n_1P_1+n_2P_2+n_3P_3+.....+n_rP_r=0$$ with not all $n_i$'s  equal to zero.
Without loss of generality, we may assume that $n_r\neq 0$. The linear dependence relation gives that
$$n_r P_r= -\sum_{i<r}{n_iP_i}\in E ( \prod_{i<r}{\tilde{K}_i^c}),$$ since each point $P_i\in E(\tilde{K}_i^c)$.
Thus we have,
$$k_r(n_r P_r)\subset \tilde{K}_r^c \cap k_r \prod_{i<r}{\tilde{K}_i^c}.$$
From this, we get that
\begin{equation}\label{divisibility}
[k_r(n_r P_r):k_r] \text{\quad divides\quad}[\tilde{K}_r^c \cap k_r \prod_{i<r}{\tilde{K}_i^c}:k_r].
\end{equation}
By \cite[Chapter8, Theorem 7]{L1}, the order of the ring class group of conductor $c$ of a number field $f$ differs from the class number of the field by the factor
$$\frac{c}{[\mathcal{O}_f^*:\mathcal{O}^*]}\prod_{p|c}\left(1-\left(\frac{f}{p}\right)\frac{1}{p} \right ),$$
where $(\frac{f}{p})$ takes value $1$ if $p$ splits completely and $-1$ if $p$ stays prime in $f$. One can bound this factor for all quadratic imaginary fields $k$ by a constant that depends only on $c$. Let's denote this constant by $Z$ for now.
Writing down all we have so far, we get:
\begin{align}\label{main3}
[\tilde {K}_r^c :k_r]\mid (\deg \Phi_{E})![k_r(P_r):k_r]\mid(\deg \Phi_{E})! M [k_r(n_rP_r):k_r].
\end{align}
Now we use\cite[Proposition 18]{RS}, which says that
$[K_r \cap k_r \prod_{i<r}{K_i}:k_r]$ is a power of $2$. Replacing each $K_i$ with $\tilde{K}_i^c$ in this expression, we obtain the right hand side of \eqref{divisibility}, hence we find that the prime-to-$2$ part of $[k_r(n_r P_r):k_r]$ must be less than $Z$. Taking the odd part of \eqref{main3}, we get an upper bound for the prime-to-$2$ part of the ring class numbers of conductor $c$  of the given quadratic imaginary fields. This bound is $(\deg\Phi_{E})! M Z.$
We conclude that if the odd part of the ring class numbers of the fields are greater than some constant $C(E,\Phi)$, then the Heegner points associated to the orders of conductor $c$ of these fields are linearly independent.
\end{proof}

\begin{rmk}
A natural question to ask would be whether we can drop the fixed conductor constraint and prove independence of Heegner points coming from orders of arbitrary conductors of quadratic imaginary fields. Although we are not going to give the proof in detail, it is not difficult to prove the following theorem.
\end{rmk}

\begin{thm}
Given an elliptic curve $E$,  fixed modular parametrization $\Phi_{E}:X_0(N)\rightarrow E$, quadratic imaginary fields $k_1,k_2, k_3, ... k_r$ , and orders of conductors of respectively $c_1,c_2,c_3,... c_r$ of these fields, there exists a lower bound $C=C(E,c,\Phi_{E})$, where $c$ is the least common multiple of $c_1,c_2,c_3, ... c_r$, such that if the odd parts of the class numbers of these fields are greater than $C$, then the Heegner points associated to the given orders of the fields are independent.
\end{thm}

The proof is similar to the one above for Theorem \ref{genorder} except that we replace the ring class field of conductor $c_i$ for the field $k_i$ with the ring ring class field of conductor $c=\text{lcm}\left(c_1,c_2, ...c_r \right)$. This is a larger field, and we can find bound on the index of it over the Hilbert class fields for all $1\leq\ i \leq r$, so that the rest of the proof is similar.

Theorem~\ref{genorder} is really a generalization of Theorem~\ref{main2} and a  result of \cite{RS}. Theorem~\ref{main2} and the result of Silverman and Rosen deal with the case where $c=1$. Note that when Heegner points are associated to the maximal orders of given quadratic imaginary fields, we have $Z=1$ and we get exactly the result as in \cite {RS} and Theorem~\ref{main2}.

\section{On the independence of Heegner points associated to real quadratic fields}

The article \cite{Darmon} by Darmon is a good reference for different types of modular parametrizations of elliptic curves. These parametrizations can be used to generate algebraic points on elliptic curves. In \cite{Darmon2} it is  explained, for an  elliptic curve $E$ over $\mathbb{Q}$, how to conjecturally assign Hegneer points to a real quadratic field $k$ for which $\sign(E,k)=-1$. Under the assumption that $\sign(E,k)=-1$, by \cite[Chapter 3 Theorem 3.17]{Darmon} there is a prime divisor $p$ of the conductor $N$ of the elliptic curve $E$ that is either inert or ramified in $k/\mathbb{Q}$. This prime is used to construct a $p$-adic uniformization. Under this parametrization, the image of the CM point is a $p$-adic point on the elliptic curve. However, it's conjectured by Darmon that this point is algebraic. Moreover he conjectures that it generates the Hilbert class field of $k$, in the strict sense over $k$.
Under the assumption that the points associated to a real quadratic field are algebraic and that the coordinates of each of them generate the Hilbert class field in the strict sense over the base field $k$, we will examine their independence. We will look separately at elliptic curves with CM and without CM.
\subsection{On non-CM elliptic curves}
One can find a bound for $[k(P): k(nP)]$ for all positive integers $n$ and number fields $k$ as long as $k(P)/k$ is abelian. To show this, Serre's theorem on the image of Galois and the linear algebra estimate in \cite{RS} are used together. The other crucial step in the proof is the strong disjointedness property. We need to compute $[K^{+}_r \cap k_r \prod_{i<r}K^{+}_{i}:k_r]$, where $K^{+}_i$ is the Hilbert class field in the strict sense of the field $k_i$ for $1\leq i \leq r$. But this number is a  power of two times the index $[K_r \cap k_r \prod_{i<r}{K_i}:k_r]$,  where $K_i$ denotes the Hilbert class field of $k_i$. Hence the odd part of $[K^{+}_r \cap k_r \prod_{i<r}K^{+}_{i}:k_r]$ is also $1$.
\subsection{On CM eliptic curves}
The proof for the CM case differs from the non-CM case at the estimation of $[k(P): k(nP)]$. In the CM case we will be using a ramification argument. This time we will be in a (potentially) larger field than the Hilbert class field, but it will not give any extra ramification at the finite places, and the same arguments will work for this case, too. So we can still use our ramification argument to say that $[k(P):k(nP)]$ is bounded for all $n$ and $k$. The rest follows similarly to the proof of Theorem~\ref{main2}.

\begin{thm}\label{gen2}
Let $E/\mathbb{Q}$ be an elliptic curve.
Assume that we are given the following:
\begin{itemize}
\item $k_1,k_2, k_3,...k_r$ are distinct quadratic real fields,
\item $h_1,h_2, h_3,...h_r$ are the  class numbers of the fields $k_1,k_2, k_3,...k_r$,
\item $P_1,P_2, P_3,...P_r$ points on the elliptic curve associated to the ring of integers of $k_1, k_2, k_3,...k_r$, respectively, according to the Heegner point construction mentioned above.
\end{itemize}
Assume that Darmon's conjectures on the properties of real Heegner points are true. Then there is a number $C$, such that if the odd parts of the numbers $h_1,h_2, h_3,...h_r$ are larger than $C$,
then the points $P_1,P_2, P_3,...P_r$ are independent in $E(\bar{\mathbb{Q}})/E_{tors}$.
\end{thm}
\begin{proof}
Let $k=k_1k_2k_3....k_r$, and let $K$ and  $K^{+}$ be respectively the Hilbert class field and the Hilbert class field in the strict sense of $k$. Similarly, let $K_i$ and $K_i^{+}$ be respectively the Hilbert class field and the Hilbert class field in the strict sense of $k_i$ for $1 \leq i \leq r$.
We assume $P_1,P_2, P_3,...P_r$ are dependent and prove an upper bound for the odd parts of $h_i$'s that only depends on $E$.
Say
$$n_1P_1+n_2P_2+n_3P_3+.....+n_rP_r=0$$ with not all $n_i$'s  equal to zero.
Without loss of generality, we may assume that $n_r\neq 0$. The linear dependence relation gives that
$$n_r P_r= -\sum_{i<r}{n_iP_i}\in E ( \prod_{i<r}{K}_i^{+}),$$ sinnce each point $P_i\in E({K}^{+}_i).$
Thus we have,
$$k_r(n_r P_r)\subset {K}_r^{+} \cap k_r \prod_{i<r}{{K}_i^{+}}.$$
and hence
\begin{equation}\label{divisibilityr}
[k_r(n_r P_r):k_r] \text{\quad divides\quad}[{K}_r^{+} \cap k_r \prod_{i<r}{{K}_i^{+}}:k_r].
\end{equation}
Remember that $[K_r \cap k_r\prod_{i<r}{K_i}:k_r]$ is a power of $2$, so the same is true of the right hand side of the relation in \eqref{divisibilityr}.
On the other hand,
\begin{align}\label{shim}
[{K}_r^{+} :k_r]=[k_r(P_r):k_r]\mid M [k_r(n_rP_r):k_r].
\end{align}
Taking the odd part of \eqref{shim} and taking in account equation \eqref{divisibilityr}, we get an upper bound for the prime-to-$2$ part of the class numbers of  of the given quadratic real fields. This bound is $M$. We conclude that for any set of distinct real quadratic fields, if the class numbers of these fields are greater than some constant number, then the Heegner points associated them will be independent.
\end{proof}

\end{document}